\documentclass[12pt,reqno]{amsart}

\newcommand\version{April 17, 2023}

\usepackage{amsmath,amsfonts,amsthm,amssymb,amsxtra}
\usepackage{bbm} 
\usepackage{bm} 
\usepackage{color}
\usepackage{hyperref}
\usepackage[normalem]{ulem}

\newtheorem{theorem}{Theorem}[section]
\newtheorem{proposition}[theorem]{Proposition}
\newtheorem{lemma}[theorem]{Lemma}
\newtheorem{corollary}[theorem]{Corollary}

\theoremstyle{definition}

\theoremstyle{remark}

\newtheorem{remark}[theorem]{Remark}

\numberwithin{equation}{section}

\newcommand{\1}{\mathbbm{1}}

\newcommand{\C}{\mathbb{C}}

\renewcommand{\epsilon}{\varepsilon}

\renewcommand{\phi}{\varphi}
\newcommand{\R}{\mathbb{R}}

\DeclareMathOperator{\dom}{dom}

\def\ds{\mathrm{d}s}
\def\dt{\mathrm{d}t}

\def\dx{\mathrm{d}x}
\def\dy{\mathrm{d}y}
\def\dz{\mathrm{d}z}
\def\balpha{\bm{\alpha}}

\def\cl{\mathcal{L}}

\newcommand{\abs}[1]{\left\lvert#1\right\rvert}
\newcommand{\norm}[1]{\left\lVert#1\right\rVert}

\newcommand{\me}[1]{\mathrm{e}^{#1}}
\newcommand{\one}{\mathbf{1}}

\makeatletter
\newcommand*{\rom}[1]{\expandafter\@slowromancap\romannumeral #1@}
\makeatother

\begin{document}

\title[Hardy operators and Sobolev norms --- \version]{Equivalence of Sobolev norms\\ involving generalized Hardy operators}

\author{Rupert L. Frank}
\address[Rupert L. Frank]{Mathematisches Institut, Ludwig-Maximilans Universit\"at M\"unchen, Theresienstr. 39, 80333 M\"unchen, Germany, and Munich Center for Quantum Science and Technology (MCQST), Schellingstr. 4, 80799 M\"unchen, Germany, and Mathematics 253-37, Caltech, Pasa\-de\-na, CA 91125, USA}
\email{rlfrank@caltech.edu}

\author{Konstantin Merz}
\address[Konstantin Merz]{Mathematisches Institut, Ludwig-Maximilans Universit\"at M\"unchen, Theresienstr. 39, 80333 M\"unchen, Germany, and Munich Center for Quantum Science and Technology (MCQST), Schellingstr. 4, 80799 M\"unchen, Germany}
\email{merz@math.lmu.de}

\author{Heinz Siedentop}
\address[Heinz Siedentop]{Mathematisches Institut, Ludwig-Maximilans Universit\"at M\"unchen, Theresienstr. 39, 80333 M\"unchen, Germany, and Munich Center for Quantum Science and Technology (MCQST), Schellingstr. 4, 80799 M\"unchen, Germany}
\email{h.s@lmu.de}

\begin{abstract}
We consider the fractional Schr\"odinger operator with Hardy potential and critical or subcritical coupling constant. This operator generates a natural scale of homogeneous Sobolev spaces which we compare with the ordinary homogeneous Sobolev spaces. As a byproduct, we obtain generalized and reversed Hardy inequalities for this operator. Our results extend those obtained recently for ordinary (non-fractional) Schr\"odinger operators and have an important application in the treatment of large relativistic atoms.
\end{abstract}

\thanks{\copyright\, 2019 by the authors. This paper may be
reproduced, in its entirety, for non-commercial purposes.\\
The authors are very grateful to an anonymous referee for many helpful remarks and suggestions. \\
Deutsche Forschungsgemeinschaft grants SI 348/15-1 (H.S.) and EXC-2111 390814868 (R.L.F., K.M., H.S.) and U.S.~National Science Foundation grant DMS-1363432 (R.L.F.) are acknowledged.}

\maketitle

\section{Introduction and main result}

\subsection*{Introduction}

The Hardy inequality is of fundamental importance in many questions in harmonic analysis, partial differential equations, spectral theory and mathematical physics. This inequality comes in various forms and the form that we will be concerned with here is that if $0<\alpha<d$, then
\begin{equation}
\label{eq:hardy}
\left\| |p|^{\alpha/2} f \right\|_{L^2(\R^d)}^2 \geq \mathcal H_{d,\alpha} \left\| |x|^{-\alpha/2} f \right\|_{L^2(\R^d)}^2
\qquad\text{for all}\ f\in C_c^\infty(\R^d)
\end{equation}
with a positive constant $\mathcal H_{d,\alpha}$. Here and in the following we use the notation
$$
|p|=\sqrt{-\Delta} \,.
$$
The optimal, that is, largest possible, value of the constant on the right side was found in \cite[Theorem 2.5]{Herbst1977} and is given by
$$
\mathcal H_{d,\alpha} =\frac{2^\alpha\Gamma((d+\alpha)/4)^2}{\Gamma((d-\alpha)/4)^2} \,.
$$
Of course, the special case $\alpha=2$, where the left side is $\| \nabla f\|^2_{L^2(\R^d)}$, had been known much longer and also the physically important special case $d=3$ and $\alpha=1$ appeared before \cite[Chapter 5, Equation (5.33)]{Kato1966}. For alternative proofs of the inequality with the sharp constant we refer to \cite{Kovalenkoetal1981,Yafaev1999,Franketal2008H,FrankSeiringer2008}.

Setting
$$
a_* := - \mathcal H_{d,\alpha} = - \frac{2^\alpha\Gamma((d+\alpha)/4)^2}{\Gamma((d-\alpha)/4)^2} \,,
$$
we infer from Hardy's inequality with optimal constant that the operator
$$
\cl_{a,\alpha} := |p|^\alpha + a|x|^{-\alpha}
\qquad\text{in}\ L^2(\R^d)
$$
is non-negative for $a\geq a_*$. More precisely, the operator $\cl_{a,\alpha}$ is defined as the Friedrichs extension of the corresponding expression on $C_c^\infty(\R^d)$. In some applications and, in particular, in our paper \cite{Franketal2020P}, the norms
$$
\left\| \cl_{a,\alpha}^{s/2} f \right\|_{L^2(\R^d)}
$$
for $a\geq a_*$ and $s>0$ appear naturally and one needs to understand the relation between these norm and those in the special case $a=0$, i.e.,
$$
\left\| |p|^{\alpha s/2} f \right\|_{L^2(\R^d)} \,.
$$
While it is relatively straightforward to see that the norms are equivalent for $s\in(0,1]$ if $a>a_*$ (see Remark \ref{ssmall}), the question becomes non-trivial for $s>1$ and was only recently solved in the remarkable paper \cite{Killipetal2018} in the local case $\alpha=2$. An important consequence of their result is that for any $a>a_*$ there is an $s_{a,d,\alpha}>1$ such that the norms are equivalent for all $0<s<s_{a,d,\alpha}$. In fact, \cite{Killipetal2018} also treat a more general problem with the $L^2(\R^d)$ norms replaced by those in $L^p(\R^d)$ for $1<p<\infty$ and discuss the dependence on $p$.

\subsection*{Main result}

The main result of our paper will be the generalization of the results from \cite{Killipetal2018} to the case $0<\alpha<2$. We restrict ourselves to the case of $L^2(\R^d)$ norms which is the one relevant for the application we have in mind.

In order to state our main result precisely, for $0<\alpha<2\wedge d$ we define
$$
\Psi_{\alpha,d}(\sigma) := -2^\alpha \frac{\Gamma(\frac{\sigma+\alpha}{2})\ \Gamma(\frac{d-\sigma}{2})}{\Gamma(\frac{d-\sigma-\alpha}{2})\ \Gamma(\frac{\sigma}{2})}
\qquad\text{if}\ \sigma\in (-\alpha,(d-\alpha)/2]\setminus\{0\}
$$
and $\Psi_{\alpha,d}(0)=0$. The function $\sigma\mapsto \Psi_{\alpha,d}(\sigma)$ is continuous and strictly decreasing with
$$
\lim_{\sigma\to-\alpha} \Psi_{\alpha,d}(\sigma) = \infty
\qquad\text{and}\qquad
\Psi_{\alpha,d}(\frac{d-\alpha}{2}) = a_* \,.
$$
(This can be shown as in \cite[Lemma 3.2]{Franketal2008H}, which treats the interval $[0,(d-\alpha)/2]$; see also the remark before  \cite[Lemma 2.3]{JakubowskiWang2018}.) As a consequence, for any $a\in[a_*,\infty)$ we can define
\begin{equation}
\label{eq:defdelta}
\delta := \Psi_{\alpha,d}^{-1}(a) \,.
\end{equation}

The following is our main result.

\begin{theorem}[Equivalence of Sobolev norms]
  \label{Thm:1.2}
  Let $\alpha\in(0,2\wedge d)$, $a\in [a_*,\infty)$ and let $\delta$ be defined by \eqref{eq:defdelta}. Let $s\in(0,2]$.
\begin{enumerate}
\item If $s<\frac{d-2\delta}{\alpha}$, then
    \begin{align}
      \label{eq:12a}
    \norm{\abs p^{\alpha\frac{s}{2}}f}_{L^2(\R^d)}\lesssim_{d,\alpha,a,s}\norm{\cl_{a,\alpha}^{\frac{s}{2}}f}_{L^2(\R^d)}
    \text{ for all } f\in C_c^\infty(\R^d) \,.
  \end{align}
\item If $s<\frac{d}{\alpha}$, then
  \begin{align}
    \label{eq:12b}
    \norm{\cl_{a,\alpha}^{\frac{s}{2}}f}_{L^2(\R^d)}\lesssim_{d,\alpha,a,s} \norm{\abs p^{\alpha\frac{s}{2}}f}_{L^2(\R^d)}
    \qquad\text{for all}\ f\in C_c^\infty(\R^d) \,.
  \end{align}
\end{enumerate}
\end{theorem}

Here and in what follows, we write $X\lesssim Y$ for non-negative quantities $X$ and $Y$, whenever there is a constant $A$ such that $X\leq A\cdot Y$. In order to emphasize that $A$ depends on some parameter $r$, we will sometimes write $X\lesssim_r Y$. Moreover, $X\sim Y$ means $Y\lesssim X\lesssim Y$. In this case, we say $X$ is \emph{equivalent} to $Y$. Moreover, in the paper we use the notation
$$
X\wedge Y:=\min\{X,Y\}
\qquad\text{and}\qquad
X\vee Y:=\max\{X,Y\} \,.
$$

\begin{remark}\label{ssmall}
The crucial point of this theorem is that it allows for some $s>1$ (unless $a=a_*$ in case of \eqref{eq:12a}). The inequalities for $s\leq 1$ can be easily seen as follows. By Hardy's inequality \eqref{eq:hardy} we have the operator inequalities
$$
\left( 1- \frac{a}{a_*} \right) |p|^\alpha \leq \mathcal L_{a,\alpha} \leq |p|^\alpha 
\ \text{if}\ a_*\leq a<0 \,,\quad
|p|^\alpha \leq \mathcal L_{a,\alpha} \leq \left( 1+ \frac{a}{a_*} \right) |p|^\alpha
\ \text{if}\ a>0 \,,
$$
and therefore, by operator monotonicity of roots (see, for instance, \cite[Theorem~2.6]{Carlen2010}), for $0<s<1$,
$$
\left( 1- \frac{a}{a_*} \right)^s\! |p|^{\alpha s} \!\leq \mathcal L_{a,\alpha}^s \leq |p|^{\alpha s} 
\ \text{if}\ a^*\leq a<0 \,,\quad
|p|^{\alpha s} \!\leq \mathcal L_{a,\alpha}^s \leq \!\left( 1+ \frac{a}{a_*} \right)^s |p|^{\alpha s}
\ \text{if}\ a>0 \,.
$$
This proves \eqref{eq:12a} and \eqref{eq:12b} for $0<s\leq 1$.
\end{remark}

\begin{remark}\label{rem:hardy}
Another crucial point of this theorem is that it covers the full range $a\geq a_*$ and $\alpha\leq 2\wedge d$. In fact, as we explain now, if $\alpha<d/2$, then there is a simple proof of a stronger bound than \eqref{eq:12b} and if, in addition, $a$ belongs to a restricted range, then there is also a simple proof of a stronger bound than \eqref{eq:12a}.

Our first claim is that if $\alpha<d/2$, then \eqref{eq:12b} holds for all $s\leq 2$. By the same argument as in the previous remark it suffices to prove this for $s=2$. In that case, we have, by \eqref{eq:hardy},
$$
\norm{\cl_{a,\alpha} f}_{L^2(\R^d)} \leq \norm{|p|^\alpha f}_{L^2(\R^d)} + |a| \norm{|x|^\alpha f}_{L^2(\R^d)}
\leq \left(1+ |a| \mathcal H_{d,2\alpha}^{-1/2} \right) \norm{|p|^\alpha f}_{L^2(\R^d)},
$$
as claimed.

To state our second claim precisely, let $a_{**}:=-\mathcal H_{d,2\alpha}^{1/2}$ (which is well-defined for $\alpha<d/2$ and is negative). We will show below that $a_{**}>a_*$. We claim that for $|a|<| a_{**}|$ inequality \eqref{eq:12b} holds for all $s\leq 2$. Again it suffices to prove this for $s=2$ and then we have, similarly as before,
$$
\norm{\cl_{a,\alpha} f}_{L^2(\R^d)} \geq \norm{|p|^\alpha f}_{L^2(\R^d)} - |a| \norm{|x|^\alpha f}_{L^2(\R^d)}
\geq \left(1- |a| \mathcal H_{d,2\alpha}^{-1/2} \right) \norm{|p|^\alpha f}_{L^2(\R^d)},
$$
as claimed.

Let us finally show that $a_{**}>a_*$. We have, in terms of $\psi=\Gamma'/\Gamma$,
\begin{align*}
\ln |a_*| - \ln |a_{**}| & = 2\ln\Gamma(\tfrac{d+\alpha}4) - 2\ln\Gamma(\tfrac{d-\alpha}4) - \ln\Gamma(\tfrac{d+2\alpha}4) + \ln\Gamma(\tfrac{d-2\alpha}4) \\
& = 2\int_{\tfrac{d-\alpha}4}^{\tfrac{d+\alpha}4} \psi(t) \,dt- \int_{\tfrac{d-2\alpha}4}^{\tfrac{d+2\alpha}4} \psi(t)\,dt \\
& = \int_{\tfrac{d}4}^{\tfrac{d+\alpha}4} (\psi(t) -\psi(t+\tfrac\alpha4))\,dt
- \int_{\tfrac{d-2\alpha}4}^{\tfrac{d-\alpha}4} (\psi(t) -\psi(t+\tfrac\alpha4))\,dt \\
& > 0 \,.
\end{align*}
The last inequality comes from the fact that both integration intervals are of length $\alpha/4$ and that the integrand is an increasing function, which in turn follows from the fact that $\psi'$ is decreasing \cite[(6.4.1)]{Davis1965}.
\end{remark}

\subsection*{Ingredients in the proof}

As a consequence of \eqref{eq:12a} and Hardy's inequality \eqref{eq:hardy} we obtain immediately a Hardy inequality for fractional powers of $\mathcal L_{a,\alpha}$.

\begin{proposition}[Generalized Hardy inequality]
 \label{Prop:1.3}
Let $\alpha\in(0,2\wedge d)$, $a\in[a_*,+\infty)$ and let $\delta$ be defined by \eqref{eq:defdelta}. Then, for any $s\in(0,\frac{d-2\delta}\alpha \wedge \frac{2d}\alpha)$,
 \begin{align}\label{eq:genhardy}
    \norm{\abs x^{-\alpha s/2}f}_2\lesssim_{d,\alpha,a,s} \norm{\cl_{a,\alpha}^{s/2}f}_2 
    \qquad\text{for all}\ f\in C_c^\infty(\R^d) \,.
 \end{align}
 Conversely, if $s\in(0,\frac{2d}\alpha\wedge\frac{2(d-2\delta)}\alpha)$ and if \eqref{eq:genhardy} holds, then $s<\frac{d-2\delta}\alpha$.
\end{proposition}

In fact, our method of proof will be different. We will not deduce \eqref{eq:genhardy} from Theorem~\ref{Thm:1.2}, but rather \eqref{eq:genhardy} will be proved first and will be an important ingredient in the proof of Theorem \ref{Thm:1.2}. Another ingredient in the proof of Theorem \ref{Thm:1.2} will be the following inequality which gives a \emph{lower} bound on the norm of $|x|^{-\alpha s/2} f$ in terms of the difference $\left(\mathcal L_{a,\alpha}^{s/2} - |p|^{\alpha s/2} \right) f$. We have not seen similar inequalities before, not even in the case $\alpha=2$ (although they can probably be extracted from the bounds on the difference of square functions in \cite{Killipetal2018}).

\begin{proposition}[Reverse Hardy inequality for differences]
\label{reversehardy}
Let $\alpha\in(0,2\wedge d)$, $a\in[a_*,+\infty)$. Then, for any $s\in(0,2]$,
$$
\left\| \left(\mathcal L_{a,\alpha}^{s/2} - |p|^{\alpha s/2} \right) f \right\|_2 \lesssim_{d,\alpha,a,s} \norm{\abs x^{-\alpha s/2}f}_2
\qquad\text{for all}\ f\in C_c^\infty(\R^d) \,.
$$
\end{proposition}    

Before proceeding, let us show that Theorem \ref{Thm:1.2} is an immediate consequence of Propositions \ref{Prop:1.3} and \ref{reversehardy}.

\begin{proof}[Proof of Theorem \ref{Thm:1.2}]
If $s<\frac{d-2\delta}{\alpha}$, we obtain by Propositions \ref{reversehardy} and \ref{Prop:1.3}
\begin{align*}
\norm{\abs p^{\alpha\frac{s}{2}}f}_{L^2(\R^d)}
& \leq \norm{\cl_{a,\alpha}^{s/2}f}_{L^2(\R^d)} + \left\| \left(\mathcal L_{a,\alpha}^{s/2} - |p|^{\alpha s/2} \right) f \right\|_2 \\
& \lesssim \norm{\cl_{a,\alpha}^{s/2}f}_{L^2(\R^d)} + \norm{\abs x^{-\alpha s/2}f}_2 \\
& \lesssim \norm{\cl_{a,\alpha}^{s/2}f}_{L^2(\R^d)} \,.
\end{align*}
This proves (1). (Note that the assumption $s<\frac{2d}{\alpha}$ in Proposition \ref{Prop:1.3} follows from $s\leq 2$ and $\alpha<d$.) If $s<\frac{d}{\alpha}$, we argue similarly, but with the Hardy inequality \eqref{eq:hardy} (with $\alpha s$ instead of $\alpha$) instead of Proposition \ref{Prop:1.3}. This proves (2).
\end{proof}

Thus, we have reduced the proof of Theorem \ref{Thm:1.2} to that of Propositions \ref{Prop:1.3} and \ref{reversehardy}. Let us explain some ideas involved in those proofs. The starting point of the analysis in \cite{Killipetal2018} for the case $\alpha=2$ (and that in \cite{Killipetal2016}) are two-sided bounds on the heat kernel of $\mathcal L_{a,2}$. Correspondingly, we rely here on very recent two-sided bounds on the heat kernel of $\mathcal L_{a,\alpha}$ for $0<\alpha<2$ from \cite{Bogdanetal2019,Choetal2018,JakubowskiWang2018}. Given these bounds on the heat kernel of $\mathcal L_{a,\alpha}$ it is a simple matter to obtain bounds on the corresponding Riesz kernels, i.e., the integral kernels of the operators $\mathcal L_{a,\alpha}^{-s/2}$. We state them explicitly as follows.

\begin{theorem}[Riesz kernels of generalized Hardy operators]
  \label{Thm:1.4}
  Let $\alpha\in(0,2\wedge d)$, $a\in [a_*,\infty)$ and let $\delta$ be defined by \eqref{eq:defdelta}. If $s\in(0,\frac{2d}\alpha\wedge\frac{2(d-2\delta)}\alpha)$, then
  \begin{align}
   \cl_{a,\alpha}^{-s/2}(x,y) \sim_{d,\alpha,a,s} \abs{x-y}^{\alpha\frac{s}{2}-d}\left(1\wedge\frac{\abs x}{\abs{x-y}}
        \wedge\frac{\abs y}{\abs{x-y}}\right)^{-\delta} \,.
  \end{align}
\end{theorem}

One of the applications of this theorem is, for instance, a proof of Proposition \ref{Prop:1.3}.

\begin{proof}[Proof of Proposition \ref{Prop:1.3}]
  Since $C_c^\infty(\R^d)\subset\dom(\cl_{a,\alpha})^{s/2}$ (which is easy to see; cf.\ the proof of \cite[Lemma~15]{FrankMerz2023}), the assertion follows from the $L^2$-boundedness of the operator $|x|^{-\alpha s/2} \mathcal L_{a,\alpha}^{-s/2}$. Using the upper bounds in Theorem \ref{Thm:1.4} this follows from the $L^2$-boundedness of the operator with integral kernel
  $$\abs x^{-\alpha\frac{s}{2}}\abs{x-y}^{\alpha\frac{s}{2}-d}
  \left(1\wedge\frac{\abs x}{\abs{x-y}}\wedge\frac{\abs y}{\abs{x-y}}\right)^{-\delta}.$$
  This follows in a straightforward way by a Schur test. Since the same argument appears already in \cite[Proposition 3.2]{Killipetal2018} (with $s$ in place of $\alpha s/2$ and $\sigma$ in place of $\delta$) and since we will perform similar Schur tests also later on in this paper, we omit the details here.
  
  The fact that \eqref{eq:genhardy} fails for $s\in[\frac{d-2\delta}{\alpha},\frac{2d}{\alpha}\wedge\frac{2(d-2\delta)}{\alpha})$ follows from the lower bound in Theorem \ref{Thm:1.4} by the same argument as in \cite[Proposition 3.2]{Killipetal2018}.
\end{proof}

This reduces the proof of Proposition \ref{Prop:1.3} to that of Theorem \ref{Thm:1.4}, which will be given in Section \ref{sec:riesz}.

While the proof of Proposition \ref{Prop:1.3} and that of Theorem \ref{Thm:1.4} rely on bounds on the individual heat kernel of $\mathcal L_{a,\alpha}$, the proof of Proposition \ref{reversehardy} relies on bounds on the difference between the heat kernels of $\mathcal L_{a,\alpha}$ and $|p|^\alpha$ and, in particular, on cancellations between both. We do not discuss these bounds in this introduction, but refer to Section \ref{sec:difference} and, in particular, to Lemma \ref{Lem:3.4}. We would like to stress here, however, that the Gaussian off-diagonal decay of the heat kernel for $\alpha=2$ is replaced by some algebraic decay for $0<\alpha<2$. Therefore the corresponding bounds are more involved than those in \cite{Killipetal2018}.

\subsection*{Applications}

We would like to end this introduction with several applications of our main result, Theorem \ref{Thm:1.2}. The most straightforward one is the following Sobolev inequality for the operator $\mathcal L_{a,\alpha}$. It is an immediate consequence of Theorem \ref{Thm:1.2} and the usual Sobolev inequality.

\begin{corollary}
Let $\alpha\in(0,2\wedge d)$, $a\in [a_*,\infty)$ and let $\delta$ be defined by \eqref{eq:defdelta}. Then, for any $s\in(0,2]$ with $s<\frac{d}{\alpha}$ if $a\geq 0$ and with $s<\frac{d-2\delta}{\alpha}$ if $a<0$,
    \begin{align}
     \left\| f \right\|_{L^\frac{2d}{d-\alpha s}(\R^d)} \lesssim_{d,\alpha,a,s} \norm{\cl_{a,\alpha}^{\frac{s}{2}}f}_{L^2(\R^d)}
    \text{ for all } f\in C_c^\infty(\R^d) \,.
  \end{align}
\end{corollary}

Our main motivation for Theorem \ref{Thm:1.2} comes from our forthcoming paper \cite{Franketal2020P} where we describe the ground state density of a large atom with relativistic effects taken into account near the nucleus. This problem was solved before in the non-relativistic case in \cite{Iantchenkoetal1996}. The relativistic case is substantially more difficult since the kinetic energy $\sqrt{|p|^2+m^2}-m$ and the potential energy $|x|^{-1}$ show for large $|p|$ the same scaling behavior. Therefore the latter one cannot be treated as a perturbation and the Sobolev spaces associated to $\mathcal L_{a,1}$ appear naturally. While we will eventually show that the problem under consideration is dominated by the small $|p|$ behavior, Theorem \ref{Thm:1.2} yields the crucial a-priori bound which controls the worst case behavior coming from large $|p|$. On a more technical level, since we are dealing with a many-body problem, we need a certain trace ideal inequality for the operator $\mathcal L_{a,1}$. Theorem~\ref{Thm:1.2} allows us to deduce such an inequality from the corresponding inequality for $|p|$. For further details, we refer to \cite{Franketal2020P}.

Our third application comes from an alternative description of an atom in the presence of relativistic effects, namely by the Coulomb--Dirac operator
\begin{align}
  D^\nu:=-i\balpha\cdot\nabla-\frac{\nu}{\abs x}
\end{align}
in $L^2(\R^3:\C^4)$ with $\balpha=(\alpha_1,\alpha_2,\alpha_3)$ being the triplet of Dirac matrices. This operator is essentially self-adjoint on
$C_c^\infty(\R^3\setminus\{0\}:\C^4)$ for $\abs\nu\leq\frac{\sqrt3}{2}$, while for $\nu\in(\frac{\sqrt3}{2},1]$ there is a distinguished choice of a self-adjoint
extension. We refer to \cite{EstebanLoss2007} for more on this and for references. In the following we always mean this operator when we write $D^\nu$. A fundamental  difference between $D^\nu$ and $\mathcal L_{a,1}$ is that the former is not bounded from below. Nevertheless, based on results of \cite{MorozovMueller2017}, we will be able to prove bounds analogous to those in Theorem~\ref{Thm:1.2}.

\begin{corollary}
Let $\nu\in(0,1]$.
\begin{enumerate}
\item If $s<1+2\sqrt{1-\nu^2}$ and $s\leq 2$, then
  \begin{align}\label{eq:dirac}
    \norm{\abs p^{\frac{s}{2}}f}_{L^2}\lesssim_{\nu,s} \norm{\abs{D^\nu}^{\frac{s}{2}}f}_{L^2}\text{ for all }f\in C_c^\infty(\R^3:\C^4) \,.
  \end{align}
\item If $s\leq 2$, then
  \begin{align}\label{eq:dirac2}
    \norm{\abs{D^\nu}^{\frac{s}{2}}f}_{L^2}\lesssim_{\nu,s} \norm{\abs p^{\frac{s}{2}}f}_{L^2}\text{ for all }f\in C_c^\infty(\R^3:\C^4).
  \end{align}
\end{enumerate}
\end{corollary}

\begin{proof}
We begin with the proof of the simpler second part, which does not rely on our main theorem and uses the same idea as in Remark \ref{rem:hardy}. We recall that $(-i\balpha\cdot\nabla)^2 = -\Delta\otimes\one_{\C^4}$. Hence by the Cauchy--Schwarz inequality and by Hardy's inequality \eqref{eq:hardy} with $d=3$ and $\alpha=2$, we find
$$
(D^\nu)^2\leq A_\nu\ (-\Delta)\otimes\one_{\C^4} \,.
$$
By operator monotonicity as in Remark \ref{ssmall} we conclude that for any $0<t\leq 1$
$$
|D^\nu|^{2t}\leq A_\nu^t\ |p|^{2t} \otimes\one_{\C^4} \,.
$$
This is \eqref{eq:dirac2} with $s=2t$.

We now turn to the significantly more difficult part of proving \eqref{eq:dirac}. It is shown in \cite{MorozovMueller2017} that
$$
(D^\nu)^2\geq A_\nu \ \mathcal L_{a,1}^2\otimes\one_{\C^4} \,.
$$
with 
$$
a= -\sqrt{1-\nu^2}\, \cot\left(\frac{\pi}{2}\sqrt{1-\nu^2}\right)
\ \text{if}\ \nu<1 \,,
\qquad
a=-\frac{2}{\pi}
\ \text{if}\ \nu=1 \,.
$$
This bound is not stated explicitly in \cite{MorozovMueller2017}, but can be easily obtained from their results. Namely, this bound restricted to the zeroth angular momentum channel is equivalent to their Lemma \rom{4}.4, whereas their Lemma \rom{4}.5 says that on the orthogonal complement of this channel, $(D^\nu)^2$ can, in fact, be bounded from below by a constant times $-\Delta\otimes\one_{\C^4}$, which in turn (by Cauchy--Schwarz and Hardy, as before) can be bounded from below by $\mathcal L_{a,1}^2\otimes\one_{\C^4}$. Again by operator monotonicity we conclude that for any $0<t<1$
$$
|D^\nu|^{2t} \geq A_\nu^t \ \mathcal L_{a,1}^{2t}\otimes\one_{\C^4} \,.
$$
Noting that $a=\Psi_{1,3}(1-\sqrt{1-\nu^2})$ we see that Theorem \ref{Thm:1.2} with $s=2t$ implies the assertion.
\end{proof}


\section{Heat and Riesz kernels}\label{sec:riesz}

Our goal in this section is to prove Theorem \ref{Thm:1.4}. We begin by recalling the recent two-sided bounds on the heat kernel of $\cl_{a,\alpha}$ from \cite{Bogdanetal2019,Choetal2018,JakubowskiWang2018}. The results in the (much simpler) case $a=0$ are classical \cite{BlumenthalGetoor1960}. We also note that in the special case $a=0$ and $\alpha=1$ an explicit formula for $\me{-t\abs p^\alpha}$ is available, see, for instance, \cite[Theorem 1.14]{SteinWeiss1971}.

\begin{theorem}[Heat kernels of generalized Hardy operators]
  \label{Thm:2.1}
Let $\alpha\in(0,2\wedge d)$, $a\in [a_*,\infty)$ and let $\delta$ be defined by \eqref{eq:defdelta}. Then the heat kernel of $\cl_{a,\alpha}$ satisfies for all $x,y\in\R^d$ and $t>0$,
    \begin{align}
      \me{-t\cl_{a,\alpha}}(x,y)
      \sim \left(1\vee \frac{t^{1/\alpha}}{\abs x}\right)^\delta
            \left(1\vee \frac{t^{1/\alpha}}{\abs y}\right)^\delta
      t^{-d/\alpha}\left(1\wedge\frac{t^{1+d/\alpha}}{\abs{x-y}^{d+\alpha}}\right).
    \end{align}
\end{theorem}

We mention a technical point related to the deduction of these results from \cite{Bogdanetal2019}. There the heat kernel, for which the bounds are shown, is defined via a perturbation series. In \cite[Theorem 5.4]{Bogdanetal2019} the authors show that the quadratic form associated to the generator of the corresponding semi-group coincides with a certain quadratic form which in \cite[Proposition 4.1]{Franketal2008H} is identified as the quadratic form of the Friedrichs extension of the operator $\cl_{a,\alpha}$ on $C_c^\infty(\R^d)$. Therefore the bounds are indeed valid for the heat kernel of $\cl_{a,\alpha}$.

From Theorem \ref{Thm:2.1} we can now derive the claimed Riesz kernel bounds.

\begin{proof}[Proof of Theorem \ref{Thm:1.4}]
By the spectral theorem the Riesz kernel can be represented as
  \begin{align}
  \label{eq:rieszproof}
  \cl_{a,\alpha}^{-s/2}(x,y)=\frac{1}{\Gamma(s/2)}\int_0^\infty\me{-t\cl_{a,\alpha}}(x,y)\,t^{s/2} \,\frac{\dt}{t} \,.
  \end{align}
Inserting the two-sided bound on $\me{-t\cl_{a,\alpha}}(x,y)$ in Theorem \ref{Thm:2.1} and changing variables we see that the right side of \eqref{eq:rieszproof} is equivalent to
    \begin{align}
    \label{eq:2.3}
    \begin{split}
      &\int_0^\infty t^{-\frac d\alpha+\frac{s}{2}}
       \left(1\vee\frac{t^{1/\alpha}}{\abs x}\right)^\delta
       \left(1\vee\frac{t^{1/\alpha}}{\abs y}\right)^\delta
       \left(1\wedge\frac{t^{1+d/\alpha}}{\abs{x-y}^{d+\alpha}}\right)\frac{\dt}{t}\\
       =& \abs{x-y}^{\alpha\frac{s}{2}-d}
       \int_0^\infty \dt\ t^\frac s2 \left( 1 \wedge t^{-\frac{d}{\alpha}-1} \right)
         \left(1\vee\frac{\abs{x-y}}{\abs x}t^{\frac{1}{\alpha}}\right)^\delta
         \left(1\vee\frac{\abs{x-y}}{\abs y}t^{\frac{1}{\alpha}}\right)^\delta \,.
    \end{split}
    \end{align}
Since $\cl_{a,\alpha}^{-s/2}(x,y)$ is symmetric in $x$ and $y$, we may and will assume that $|x|\leq |y|$. Therefore, we need to show that the integral on the right side of \eqref{eq:2.3} is equivalent to
    \begin{align*}
      \left(1\wedge\frac{\abs x}{\abs{x-y}}\right)^{-\delta}
    \end{align*}
for $|x|\leq |y|$. To do this, we distinguish whether $|x-y|\leq 4|x|$ or not.

\medskip

\emph{Case $|x-y|\leq 4|x|$.} In this case we have $|y|\leq |x|+|x-y|\leq 5|x|$ and therefore
$$
|x-y|\lesssim |x| \sim |y| \,.
$$
Thus, the integral on the right side of \eqref{eq:2.3} is equivalent to
$$
\int_0^\infty \dt\ t^\frac s2 \left( 1 \wedge t^{-\frac{d}{\alpha}-1} \right)
         \left(1\vee \lambda^{-1} t^{\frac{1}{\alpha}}\right)^{2\delta}
$$
with $\lambda = |x|/|x-y|\geq 1/4$, and we need to show that this is equivalent to $1$. We have
\begin{align*}
\int_0^{\lambda^\alpha} \dt\ t^\frac s2 \left( 1 \wedge t^{-\frac{d}{\alpha}-1} \right)
         \left(1\vee \lambda^{-1} t^{\frac{1}{\alpha}}\right)^{2\delta}
& = \int_0^{\lambda^\alpha} \dt\ t^\frac s2 \left( 1 \wedge t^{-\frac{d}{\alpha}-1} \right) \sim 1 \,,
\end{align*}
since $\lambda\geq 1/4$ and since the integral converges according to the assumption $\frac s2 <\frac d\alpha$. On the other hand, using again $\lambda\geq 1/4$,
\begin{align*}
\int_{\lambda^\alpha}^\infty \dt\ t^\frac s2 \left( 1 \wedge t^{-\frac{d}{\alpha}-1} \right) \left(1\vee \lambda^{-1} t^{\frac{1}{\alpha}}\right)^{2\delta}
& = \int_{\lambda^\alpha}^\infty \dt\ t^\frac s2 \left( 1 \wedge t^{-\frac{d}{\alpha}-1} \right) \lambda^{-2\delta} t^\frac{2\delta}{\alpha} \\
& \lesssim \int_{\lambda^\alpha}^\infty \dt\ t^{\frac{s}{2}-\frac{d}{\alpha}-1} \lambda^{-2\delta} t^\frac{2\delta}{\alpha} \sim \lambda^{\frac{s\alpha}{2} -d} \lesssim 1 \,,
\end{align*}
where we used $\frac{s}{2}+\frac{2\delta}{\alpha}<\frac{d}{\alpha}$ for the convergence of the integral, as well as $\frac{s\alpha}{2}<d$. This proves the claimed upper bound. Since the integral from zero to $\lambda^\alpha$ is bounded away from zero and that from $\lambda^\alpha$ to infinity is non-negative, we also obtain the claimed lower bound.

\medskip

\emph{Case $|x-y|\geq 4|x|$.} In this case we have both $|x-y|\leq |x|+|y|\leq 2|y|$ and $|y|\leq |x|+|x-y|\leq \frac54 |x-y|$ and therefore
$$
|x|\leq |y|\sim |x-y| \,.
$$
Thus, the integral on the right side of \eqref{eq:2.3} is equivalent to
$$
\int_0^\infty \dt\ t^\frac s2 \left( 1 \wedge t^{-\frac{d}{\alpha}-1} \right)
         \left(1\vee \lambda^{-1} t^{\frac{1}{\alpha}}\right)^{\delta} \left(1\vee t^{\frac{1}{\alpha}}\right)^{\delta}
$$
with $\lambda = |x|/|x-y|\leq 1/4$, and we need to show that this is equivalent to $\lambda^{-\delta}$. We have, using $\lambda\leq 1/4$ and $\frac{s}{2}+\frac{\delta}{\alpha}>-1$,
\begin{align*}
\int_{\lambda^\alpha}^1 \dt\ t^\frac s2 \left( 1 \wedge t^{-\frac{d}{\alpha}-1} \right) \left(1\vee \lambda^{-1} t^{\frac{1}{\alpha}}\right)^{\delta} \left(1\vee t^{\frac{1}{\alpha}}\right)^{\delta}
& = \int_{\lambda^\alpha}^1 \dt\ t^{\frac{s}{2}} \lambda^{-\delta} t^\frac{\delta}{\alpha} \sim \lambda^{-\delta}
\end{align*}
and, using $\frac{s}{2}+\frac{2\delta}{\alpha}<\frac{d}{\alpha}$,
$$
\int_1^\infty \dt\ t^\frac s2 \left( 1 \wedge t^{-\frac{d}{\alpha}-1} \right) \left(1\vee \lambda^{-1} t^{\frac{1}{\alpha}}\right)^{\delta} \left(1\vee t^{\frac{1}{\alpha}}\right)^{\delta}
= \int_1^\infty \dt\ t^{\frac{s}{2}-\frac{d}{\alpha}-1} \lambda^{-\delta} t^{\frac{\delta}{\alpha}} t^{\frac{\delta}{\alpha}} \sim \lambda^{-\delta} \,.
$$
On the other hand,
$$
\int_0^{\lambda^\alpha} \dt\ t^\frac s2 \left( 1 \wedge t^{-\frac{d}{\alpha}-1} \right) \left(1\vee \lambda^{-1} t^{\frac{1}{\alpha}}\right)^{\delta} \left(1\vee t^{\frac{1}{\alpha}}\right)^{\delta}
= \int_0^{\lambda^\alpha} \dt\ t^{\frac{s}{2}} \sim \lambda^{\frac{s\alpha}{2}+\alpha} \lesssim \lambda^{-\delta} \,.
$$
These estimates yield the claimed upper and lower bounds. This completes the proof of the theorem.
  \end{proof}


\section{Difference of heat kernels}\label{sec:difference}

Our goal in this section is to prove Proposition \ref{reversehardy}. A key tool for this will be a bound on the difference of the heat kernels of $\mathcal L_{a,\alpha}$ and $|p|^\alpha$, namely,
$$
K_t^\alpha(x,y):= \me{-t |p|^\alpha}(x,y) - \me{-t\cl_{a,\alpha}}(x,y) \,.  
$$
We recall from the previous section that we have bounds on the individual kernels $\me{-t |p|^\alpha}(x,y)$ and $\me{-t\cl_{a,\alpha}}(x,y)$. The following lemma shows that there is a cancellation, coming from taking the difference, in the region $(|x|\vee |y|)^\alpha\geq t$ and $|x|\sim |y|$.

We will formulate the bound in terms of the functions
$$
L_t^\alpha(x,y) := \1_{\{(|x|\vee|y|)^\alpha\leq t\}} t^{-\frac{d}{\alpha}} \left( \frac{t^{2/\alpha}}{|x||y|} \right)^{\delta_+}
+ \1_{\{ (|x|\vee|y|)^\alpha\geq t\}} \frac{t}{(\abs x\vee\abs y)^{d+\alpha}} \left( 1 \vee \frac{t^{1/\alpha}}{|x|\wedge|y|} \right)^{\delta_+}
$$
and
$$
M_t^\alpha(x,y) := \1_{\{(|x|\vee|y|)^\alpha\geq t\}} \1_{\{\frac12 |x|\leq |y|\leq 2|x|\}}
\frac{t^{1-\frac{d}{\alpha}}}{(|x|\wedge|y|)^\alpha} \left( 1 \wedge \frac{t^{1+\frac{d}{\alpha}}}{|x-y|^{d+\alpha}} \right).
$$
Here $\delta_+=\max\{\delta,0\}$, that is, $\delta_+=0$ if $a\geq 0$ and
$\delta_+=\delta$ if $a<0$.

\begin{lemma}[Difference of kernels]\label{Lem:3.4}
Let $\alpha\in(0,2\wedge d)$, $a\in[a_*,\infty)$ and let $\delta$ be defined by \eqref{eq:defdelta}. Then for all $x,y\in\R^d$ and $t>0$
  \begin{align}
    \label{eq:4.7}
    \abs{K_t^\alpha(x,y)}\lesssim
L_t^\alpha(x,y) + M_t^\alpha(x,y) \,.
  \end{align}
\end{lemma}

\begin{proof}
By scaling it suffices to consider $t=1$ and by symmetry it suffices to consider $|x|\leq |y|$. We will drop the subscript $t$ in $K_t^\alpha$, $L_t^\alpha$ and $M_t^\alpha$.

If $a\geq 0$ we combine the maximum principle for the heat equation and the bound from Theorem \ref{Thm:2.1} for $a=0$ and obtain
$$
0 \leq K^\alpha(x,y) \leq \me{-\abs p^\alpha}(x,y)
     \sim 1\wedge \abs{x-y}^{-d-\alpha} \,.
$$
This proves the bound $K^\alpha(x,y)\lesssim L^\alpha(x,y)$ if $|y|\leq 1$ (by bounding the minimum by $1$) or if $|y|\geq 1$ and $|x|\leq (1/2)|y|$ (by bounding the minimum by $|x-y|^{-d-\alpha} \lesssim |y|^{-d-\alpha}$).

If $a<0$, then again by combining the maximum principle and Theorem \ref{Thm:2.1} we find
$$
0 \leq - K^\alpha(x,y) \leq \me{-\cl_{a,\alpha}}(x,y)
\lesssim \left(1\vee \abs x^{-\delta}\right) \left(1\vee \abs y^{-\delta}\right)
    \left(1\wedge \abs{x-y}^{-d-\alpha} \right).
$$
Again this gives the bound $-K^\alpha(x,y)\lesssim L^\alpha(x,y)$ if $|y|\leq 1$ (because then the product of the first two factors on the right side is $\leq (|x||y|)^{-\delta}$) or if $|y|\geq 1$ and $|x|\leq (1/2) |y|$ (because then the smaller one of the first two factors equals $1$).

Thus, from now on we assume that $|y|\geq 1$ and $(1/2)|y|\leq |x|\leq |y|$. By Duhamel's formula we have
  \begin{align}\label{eq:duhamel}
\me{-\abs p^\alpha}-\me{-\cl_{a,\alpha}} = a\int_0^1\ds\ \me{-(1-s)\abs p^\alpha}\abs x^{-\alpha}\me{-s\cl_{a,\alpha}} \,.
  \end{align}
If $a\geq 0$ we use again the maximum principle and Theorem \ref{Thm:2.1} (with $a=0$) to conclude that
$$
0\leq K^\alpha(x,y) \lesssim \int_0^1\ds\int_{\R^d}\dz\, \abs z^{-\alpha}s^{-\frac{d}{\alpha}}(1-s)^{-\frac{d}{\alpha}}
 \left(1\wedge\frac{(1-s)^{1+\frac{d}{\alpha}}}{\abs{x-z}^{d+\alpha}}\right)
 \left(1\wedge\frac{s^{1+\frac{d}{\alpha}}}{\abs{y-z}^{d+\alpha}}\right),
$$
and if $a<0$ we use Theorem \ref{Thm:2.1} to conclude that
  \begin{align*}
0\leq - K^\alpha(x,y) & \lesssim \int_0^1\ds
  \int_{\R^d}\dz\ s^{-\frac{d}{\alpha}}(1-s)^{-\frac{d}{\alpha}} \abs z^{-\alpha}\left(1\wedge\frac{(1-s)^{1+\frac{d}{\alpha}}}{\abs{x-z}^{d+\alpha}}\right)\\
   & \qquad\qquad\qquad\qquad \times\left(1 + \frac{s^{\frac{\delta}{\alpha}}}{\abs y^\delta}\right)
    \left(1+\frac{s^{\frac{\delta}{\alpha}}}{\abs z^\delta}\right)
    \left(1\wedge\frac{s^{1+\frac{d}{\alpha}}}{\abs{y-z}^{d+\alpha}}\right) \\
& \lesssim \int_0^1\ds
  \int_{\R^d}\dz\ s^{-\frac{d}{\alpha}}(1-s)^{-\frac{d}{\alpha}} \abs z^{-\alpha}\left(1\wedge\frac{(1-s)^{1+\frac{d}{\alpha}}}{\abs{x-z}^{d+\alpha}}\right)\\
   & \qquad\qquad\qquad\qquad \times \left(1+\frac{s^{\frac{\delta}{\alpha}}}{\abs z^\delta}\right)
    \left(1\wedge\frac{s^{1+\frac{d}{\alpha}}}{\abs{y-z}^{d+\alpha}}\right).
  \end{align*}
The second inequality here comes from the fact that $s^{1/\alpha}\leq 1\leq |y|$.

Thus, it remains to bound 
\begin{equation}
\label{eq:integralkbound}
\int_0^1\ds \int_{\R^d}\dz\ \ s^{-\frac{d}{\alpha}}(1-s)^{-\frac{d}{\alpha}} \abs z^{-\alpha}\left(1\wedge\frac{(1-s)^{1+\frac{d}{\alpha}}}{\abs{x-z}^{d+\alpha}}\right) \left(1+\frac{s^{\frac{\delta_+}{\alpha}}}{\abs z^{\delta_+}}\right)
    \left(1\wedge\frac{s^{1+\frac{d}{\alpha}}}{\abs{y-z}^{d+\alpha}}\right).
\end{equation}
We do this by dividing the $z$ integration into two parts and we begin with the part where $|z| \geq \frac12 |x|$. In this case, we can bound
$$
|z|^{-\alpha} \lesssim |x|^{-\alpha}
\qquad\text{and}\qquad
s^\frac{1}{\alpha} \leq 1 \leq |y| \lesssim |x| \lesssim |z| \,.
$$
Applying these bounds and then enlarging the $z$ integration to all of $\R^d$ we find that the integral in question is bounded by a constant times
\begin{align*}
& \frac{1}{|x|^\alpha} \int_0^1\ds \int_{\R^d}\dz\ \ s^{-\frac{d}{\alpha}}(1-s)^{-\frac{d}{\alpha}}  \left(1\wedge\frac{(1-s)^{1+\frac{d}{\alpha}}}{\abs{x-z}^{d+\alpha}}\right) \left(1\wedge\frac{s^{1+\frac{d}{\alpha}}}{\abs{y-z}^{d+\alpha}}\right) \\
& \lesssim \frac{1}{|x|^\alpha} \int_0^1\ds \int_{\R^d}\dz\ e^{-(1-s)|p|^\alpha}(x,z) e^{-s|p|^\alpha}(z,y)  \\
& = \frac{1}{|x|^\alpha} \int_0^1\ds\ e^{-|p|^\alpha}(x,y) \\
& \lesssim \frac{1}{|x|^\alpha} \left( 1\wedge |x-y|^{-d-\alpha} \right) = M^\alpha(x,y) \,,
\end{align*}
Here we used Theorem \ref{Thm:2.1} twice (with $a=0$) together with the semi-group property.

It remains to bound the part of the integral in \eqref{eq:integralkbound} corresponding to $|z|\leq\frac12|x|$. We bound
\begin{align*}
& \int\limits_{\abs z\leq\frac12\abs x} \dz\ \abs z^{-\alpha} \frac{1}{\abs{x-z}^{d+\alpha}} 
\left(1 + \frac{s^{\frac{\delta_+}{\alpha}}}{\abs z^{\delta_+}}\right)
\frac{1}{\abs{y-z}^{d+\alpha}} \\
& \quad \lesssim \frac{1}{|x|^{d+\alpha} |y|^{d+\alpha}} \int\limits_{\abs z\leq\frac12 \abs x} \dz\ \abs z^{-\alpha} \left(1 + \frac{s^{\frac{\delta_+}{\alpha}}}{\abs z^{\delta_+}}\right) \\
& \quad \sim \frac{1}{|x|^{d+\alpha} |y|^{d+\alpha}} \left( |x|^{d-\alpha} + s^\frac{\delta_+}{\alpha} |x|^{d-\alpha-\delta_+} \right) 
\sim \frac{1}{|x|^{d+\alpha}|y|^{d+\alpha}} \ |x|^{d-\alpha} \,.
\end{align*}
In the next to last step we used $\alpha +\delta_+<d$ (which follows from $\delta\leq (d-\alpha)/2$ and $\alpha< d$) and in the last step we used again $s^{1/\alpha}\leq 1\leq |y| \lesssim |x|$. Inserting this into the above bound we obtain
\begin{align*}
& \int_0^1\ds\ s^{-\frac{d}{\alpha}}(1-s)^{-\frac{d}{\alpha}}
  \int\limits_{\abs z\leq\frac12\abs x}\dz\ \abs z^{-\alpha}\left(1\wedge\frac{(1-s)^{1+\frac{d}{\alpha}}}{\abs{x-z}^{d+\alpha}}\right)\\
   & \qquad\qquad\qquad\qquad\qquad \times
    \left(1+\frac{s^{\frac{\delta_+}{\alpha}}}{\abs z^{\delta_+}}\right)
    \left(1\wedge\frac{s^{1+\frac{d}{\alpha}}}{\abs{y-z}^{d+\alpha}}\right) \\
& \quad \lesssim \frac{|x|^{d-\alpha}}{|x|^{d+\alpha}|y|^{d+\alpha}} \int_0^1\ds\ s(1-s) \\
& \quad \sim \frac{|x|^{d-\alpha}}{|x|^{d+\alpha}|y|^{d+\alpha}}
\lesssim \frac{1}{|y|^{d+\alpha}} \leq L^\alpha(x,y) \,.
\end{align*}
In the next to last step we used again $1\leq |y|\lesssim |x|$. This proves the claimed bound.
\end{proof}

Finally, we can complete the proof of our main result by giving the

\begin{proof}[Proof of Proposition \ref{reversehardy}]
Since for $s=2$ the inequality is trivial, we may assume that $0<s<2$. In that case
$$
\int_0^\infty \frac{\dt}{t}\, t^{-s/2} (e^{-t} - 1) = \Gamma(-s/2) \,,
$$
and therefore, by the spectral theorem and by scaling,
\begin{align*}
\left(\mathcal L_{a,\alpha}^{s/2} - |p|^{\alpha s/2} \right) f 
& = - \frac{1}{\Gamma(-s/2)} \int_0^\infty \frac{\dt}{t}\ t^{-s/2} \left( \left( e^{-t|p|^\alpha} - 1 \right) f - \left( e^{-t\mathcal L_{a,\alpha}}-1\right) f \right) \notag \\
& = - \frac{1}{\Gamma(-s/2)} \int_0^\infty \frac{\dt}{t}\ t^{-s/2} \left( e^{-t|p|^\alpha} - e^{-t\mathcal L_{a,\alpha}}\right) f  \notag \\
& = - \frac{1}{\Gamma(-s/2)} \int_0^\infty \frac{\dt}{t}\ t^{-s/2} \int_{\R^d} \dy\ K_t^\alpha(\cdot,y) f(y) \,.
\end{align*}
Therefore, by Lemma \ref{Lem:3.4} we can bound
\begin{align}\label{eq:schur}
\left\| \left(\mathcal L_{a,\alpha}^{s/2} - |p|^{\alpha s/2} \right) f \right\|_2
& \lesssim \norm{\int_{\R^d}\dy\ \int_0^\infty \frac{\dt}{t}\ t^{-\frac{s}{2}}L_t^\alpha(x,y)\abs y^{\alpha\frac{s}{2}} g(y)}_2 \notag\\
& \quad +\norm{\int_{\R^d}\dy\ \int_0^\infty \frac{\dt}{t}\ t^{-\frac{s}{2}} M_t^\alpha(x,y) \abs y^{\alpha\frac{s}{2}}g(y)}_2 \,,
\end{align}
where we abbreviate $g(y) := |y|^{-\alpha s/2} |f(y)|$. Our goal is to bound both terms on the right side of \eqref{eq:schur} by a constant times $\|g\|_2$.

We begin with the first term and compute
\begin{align*}
\int_0^\infty \frac{\dt}{t}\ t^{-\frac{s}{2}}L_t^\alpha(x,y)
& = (|x||y|)^{-\delta_+} \int\limits_{t\geq(\abs x\vee\abs y)^\alpha}\frac{\dt}{t}\ t^{-\frac{s}{2} - \frac{d-2\delta_+}{\alpha}} \\
& \quad
+ \int\limits_{t\leq(\abs x\vee\abs y)^\alpha}\frac{\dt}{t}\, t^{-\frac{s}{2}}\, \frac{t}{(|x|\vee|y|)^{d+\alpha}} \left( 1 \vee \frac{t^{1/\alpha}}{|x|\wedge|y|} \right)^{\delta_+} \\
& \sim \frac{1}{(|x|\vee|y|)^{\frac{s\alpha}2 +d}} \left( \frac{|x|\vee|y|}{|x|\wedge|y|} \right)^{\delta_+} .
\end{align*}
Here we used the fact that $\frac{s}{2} + \frac{d-2\delta_+}{\alpha}>0$ (which follows from $\delta\leq(d-\alpha)/2$.) Thus,
$$
\norm{\int_{\R^d}\dy\ \int_0^\infty \frac{\dt}{t}\ t^{-\frac{s}{2}}L_t^\alpha(x,y) \abs y^{\alpha\frac{s}{2}}g(y)}_2
\lesssim \norm{\int_{\R^d}\dy\ \frac{1}{(|x|\vee|y|)^{d}} \left( \frac{|x|\vee|y|}{|x|\wedge|y|} \right)^{\delta_+} g(y)}_2 \,.
$$
For any $\delta_+<\beta<d-\delta_+$ (such $\beta$ exist since $\delta\leq (d-\alpha)/2<d/2$) we have
$$
\sup_{y\in\R^d} \int_{\R^d} \dx\, \left( \frac{|y|}{|x|} \right)^\beta \frac{1}{(|x|\vee|y|)^{d}} \left( \frac{|x|\vee|y|}{|x|\wedge|y|} \right)^{\delta_+} =\int_{\R^d} \frac{\dz}{|z|^\beta (|z|\vee 1)^{d}} \left( \frac{|z|\vee 1}{|z|\wedge 1} \right)^{\delta_+} <\infty \,,
$$
and therefore by a Schur test with weights we conclude that
$$
\norm{\int_{\R^d}\dy\ \frac{1}{(|x|\vee|y|)^{d}} \left( \frac{|x|\vee|y|}{|x|\wedge|y|} \right)^{\delta_+} g(y)}_2 \lesssim \|g\|_2 \,.
$$
This shows that the first term in \eqref{eq:schur} satisfies the claimed bound.

We now turn to the second term in \eqref{eq:schur}. Since $|x|\sim |y|$ on the support of this kernel, we have
$$
\norm{\int_{\R^d} \dy \int_0^\infty \frac{\dt}{t}\, t^{-\frac{s}{2}} \, M_t^\alpha(x,y) |y|^\frac{\alpha s}{2} g(y) }
\lesssim \norm{\int_{\R^d} \dy \int_0^\infty \frac{\dt}{t}\, t^{-\frac{s}{2}} \, M_t^\alpha(x,y) (|x||y|)^\frac{\alpha s}{4} g(y) }.
$$
This replaces the kernel by a symmetric one and we only have to perform a single Schur test (instead of two). We have
\begin{align*}
& \sup_{y\in\R^d} \int_{\R^d} \dx \int_0^\infty \frac{\dt}t\ t^{-\frac{s}{2}}\, M_t^\alpha(x,y)  (|x||y|)^\frac{\alpha s}{4} \\
& = \sup_{y\in\R^d} \int\limits_{\frac12 |y|\leq |x|\leq 2|y|} \dx \int\limits_{t\leq(|x|\vee|y|)^\alpha} \frac{\dt}t\ t^{-\frac{s}{2}}\, \frac{t^{1-\frac d\alpha}}{(|x|\wedge |y|)^\alpha} \left( 1 \wedge \frac{t^{1+\frac d\alpha}}{|x-y|^{d+\alpha}} \right)  (|x||y|)^\frac{\alpha s}{4} \\
& \lesssim \sup_{y\in\R^d} |y|^{\frac{\alpha s}{2}-\alpha} \int\limits_{\frac12 |y|\leq |x|\leq 2|y|} \dx \int\limits_{t\leq(2|y|)^\alpha} \frac{\dt}t\ t^{-\frac{s}{2} + 1 - \frac{d}{\alpha}} \, \left( 1 \wedge \frac{t^{1+\frac d\alpha}}{|x-y|^{d+\alpha}} \right)  \,.
\end{align*}
We now interchange the order of integration and do the $x$ integral first. We bound
\begin{align*}
\int\limits_{\frac12 |y|\leq |x|\leq 2|y|} \dx \left(1\wedge\frac{t^{1+\frac{d}{\alpha}}}{\abs{x-y}^{d+\alpha}}\right)
\leq \int_{\R^d} \dx \left(1\wedge\frac{t^{1+\frac{d}{\alpha}}}{\abs{x-y}^{d+\alpha}}\right) \sim t^\frac{d}{\alpha} \,.
\end{align*}
Therefore, the supremum above is bounded by a constant times
$$
\sup_{y\in\R^d} |y|^{\frac{\alpha s}{2}-\alpha} \int\limits_{t\leq(2|y|)^\alpha} \frac{\dt}t\ t^{-\frac{s}{2} + 1} \sim 1 \,.
$$
Thus, the Schur test yields
$$
\norm{\int_{\R^d} \dy \int_0^\infty \frac{\dt}{t}\, t^{-\frac{s}{2}} \, M_t^\alpha(x,y) (|x||y|)^\frac{\alpha s}{4} g(y)}
 \lesssim \|g\|_2 \,.
$$
This concludes the proof of the proposition.
\end{proof}


\section{A generalization}

In our application in \cite{Franketal2020P} we will need a slight generalization of the bounds that we have derived so far to the case of not necessarily power-like potentials. More precisly, we consider functions $V$ on $\R^d$ satisfying
\begin{equation}\label{eq:u}
\frac{a}{|x|^\alpha} \leq V(x) \leq \frac{\tilde a}{|x|^\alpha}
\end{equation}
with parameters $a_*\leq a\leq \tilde a<\infty$ and we prove the following result.

\begin{theorem}
  \label{Thm:1.2gen}
Let $\alpha\in(0,2\wedge d)$, $a_*\leq a\leq \tilde a<\infty$ and let $\delta$ be defined by \eqref{eq:defdelta}. Let $s\in (0,2]$.
\begin{enumerate}
\item[(1)] If $s<\frac{d-2\delta}{\alpha}$, then for any $V$ satisfying \eqref{eq:u},
\begin{align}
      \label{eq:12agen}
    \norm{\abs p^{\alpha s/2}f}_{L^2(\R^d)}\lesssim_{d,\alpha,a,s}\norm{(|p|^\alpha+V)^{s/2}f}_{L^2(\R^d)}
    \text{ for all } f\in C_c^\infty(\R^d) \,.
  \end{align}
\item[(2)] If $s<\frac{d}{\alpha}$, then for any $V$ satisfying \eqref{eq:u},
  \begin{align}
    \label{eq:12bgen}
    \norm{(|p|^\alpha+V)^{s/2}f}_{L^2(\R^d)}\lesssim_{d,\alpha,a,s} \norm{\abs p^{\alpha s/2}f}_{L^2(\R^d)}
    \qquad\text{for all}\ f\in C_c^\infty(\R^d) \,.
  \end{align}
\end{enumerate}
\end{theorem}

We emphasize that $\delta$ is defined with respect to $a$ (and not with respect to $\tilde a$). It is interesting to note that the constants in Theorem \ref{Thm:1.2gen} are independent of $\tilde a$. By an approximation argument this would allow us to extend the theorem to a larger class of potentials, but we will not need this.

The proof of Theorem \ref{Thm:1.2gen} follows immediately from the following two propositions in the same way as Theorem \ref{Thm:1.2} followed from Propositions \ref{Prop:1.3} and \ref{reversehardy}.

\begin{proposition}
 \label{Prop:1.3gen}
Let $\alpha\in(0,2\wedge d)$, $a_*\leq a\leq \tilde a<\infty$ and let $\delta$ be defined by \eqref{eq:defdelta}. Then, for any $s\in(0,\frac{d-2\delta}{\alpha}\wedge\frac{2d}{\alpha})$ and for any $V$ satisfying \eqref{eq:u},
 \begin{align*}
    \norm{\abs x^{-\alpha s/2}f}_2\lesssim_{d,\alpha,a,s} \norm{(|p|^\alpha+V)^{s/2}f}_2 
    \qquad\text{for all}\ f\in C_c^\infty(\R^d) \,.
 \end{align*}
\end{proposition}

\begin{proof}
According to the maximum principle, we have for all $x,y\in\R^d$ and $t>0$,
\begin{equation}
\label{eq:maxprincu}
0\leq e^{-t(|p|^\alpha+V)}(x,y) \leq e^{-t\cl_{a,\alpha}}(x,y) \,,
\end{equation}
and therefore, by the analogue of \eqref{eq:rieszproof}, also
$$
(|p|^\alpha+V)^{-s/2}(x,y) \leq \cl_{a,\alpha}^{-s/2}(x,y) \,.
$$
This implies that the upper bounds in Theorem \ref{Thm:1.4} remain valid for $(|p|^\alpha +V)^{-s/2}$. The proposition now follows in the same way as Proposition \ref{Prop:1.3}. 
\end{proof}

\begin{proposition}
\label{reversehardygen}
Let $\alpha\in(0,2\wedge d)$, $a_*\leq a\leq \tilde a<\infty$ and $s\in(0,2]$. For any $V$ satisfying \eqref{eq:u},
$$
\left\| \left( \left( |p|^\alpha +V\right)^{s/2} - |p|^{\alpha s/2} \right) f \right\|_2 \lesssim_{d,\alpha,a,s} \norm{\abs x^{-\alpha s/2}f}_2
\qquad\text{for all}\ f\in C_c^\infty(\R^d) \,.
$$
\end{proposition}    

\begin{proof}
Let
$$
\tilde K_t^\alpha(x,y) := e^{-t|p|^\alpha}(x,y) - e^{-t(|p|^\alpha +V)}(x,y) \,.
$$
According to the maximum principle, we have for any $x,y\in\R^d$ and $t>0$,
$$
e^{-t|p|^\alpha}(x,y) - e^{-t\cl_{a,\alpha}}(x,y) \leq 
\tilde K_t^\alpha(x,y) \leq e^{-t|p|^\alpha}(x,y) - e^{-t\cl_{\tilde a,\alpha}}(x,y) \,.
$$
Therefore Lemma \ref{Lem:3.4} with $a$ and $\tilde a$ implies that the corresponding statement also holds for $\tilde K_t^\alpha$ in place of $K_t^\alpha$. (Here we use the fact that $\tilde\delta$, defined by \eqref{eq:defdelta} with $\tilde a$ in place of $a$, satisfies $\tilde\delta\leq\delta$.) With the analogue of Lemma \ref{Lem:3.4} at hand, the proof of the proposition follows in the same way as Proposition~\ref{reversehardy}.
\end{proof}

\appendix
\section{Extension of Theorem~\ref{Thm:1.2}}

In this appendix\footnote{This appendix does not appear in the published version of this paper.}, we show that the inequalities in Theorem~\ref{Thm:1.2} actually hold not only for functions in $C^\infty_c(\R^d)$, but for all functions in the domains of $(\cl_{a,\alpha})^{\frac s2}$ and $|p|^{\alpha s/2}$, respectively. As a consequence, we obtain the operator inequalities
$$
|p|^{\alpha s} \lesssim \mathcal L_{\alpha,a}^s \,,
\qquad\text{resp.}\qquad
\mathcal L_{\alpha,a}^s \lesssim |p|^{\alpha s} \,,
$$
which are used in our work \cite{Franketal2020P}. The precise statement is the following.

\begin{theorem}\label{Thm:1.2density}
  Let $\alpha\in(0,2\wedge d)$, $a\in [a_*,\infty)$ and let $\delta$ be defined by \eqref{eq:defdelta}. Let $s\in(0,2]$.
  \begin{enumerate}
  \item If $s<(d-2\delta)/\alpha$, then $\dom (\cl_{a,\alpha})^{s/2}\subset\dom |p|^{\alpha s/2}$ and \eqref{eq:12a} holds for all $f\in\dom (\cl_{a,\alpha})^{\frac s2}$. Moreover $C_c^\infty$ is an operator core for $\cl_{a,\alpha}^{\frac s2}$.

  \item If $s<\frac{d}{\alpha}$, then $\dom |p|^{\frac{\alpha s}{2}}\subset \dom (\cl_{a,\alpha})^{\frac{s}{2}}$ and \eqref{eq:12b} holds for all $f\in\dom |p|^{\frac{\alpha s}{2}}$. Moreover, $C_c^\infty$ is an operator core for $|p|^{\frac{\alpha s}{2}}$.
  \end{enumerate}
\end{theorem}

The proof of this theorem is follows closely that of \cite[Theorems~1 and 24]{FrankMerz2023}, which, in turn uses ideas of \cite[Lemma~4.4]{Killipetal2016}. Therefore we will be brief and only explain the main changes.

The proof is based on the following two ingredients. The first one are pointwise bounds for functions in $\me{-t\cl_{a,\alpha}}C_c^\infty(\R^d)$ and on their local H\"older seminorms. For a function $u$ on a set $\Omega$ and $0<\beta\leq 2$, we write
\begin{equation}
	\label{eq:holder}
	[u]_{C^\beta(\Omega)} :=
	\begin{cases}
		\sup_{x,y\in\Omega} \frac{|u(x)-u(y)|}{|x-y|^\beta} & \text{if}\ 0<\beta\leq 1 \,, \\
		\sup_{x,y\in\Omega} \frac{|\nabla u(x)-\nabla u(y)|}{|x-y|^{\beta-1}} & \text{if}\ 1<\beta\leq 2 \,.
	\end{cases}
\end{equation}

\begin{lemma}\label{pointwise}
	Let $\alpha$, $a$, and $\delta$ be as in Theorem~\ref{Thm:1.2density}. Let $0<t<\infty$ and $\psi\in e^{-t\cl_{a,\alpha}}C_c^\infty(\R^d)$. Then, for all $x\in\R^d$,
	\begin{align}
		\label{eq:heatbound1}
		|\psi(x)| & \lesssim |x|^{-\delta} \wedge |x|^{-d-\alpha} \,, \\
		\label{eq:heatbound2}
		|\cl_{a,\alpha}\psi(x)| & \lesssim |x|^{-\delta} \wedge |x|^{-d-\alpha} \,, \\
		\label{eq:heatbound3}
		||p|^{\alpha} \psi(x)| & \lesssim (1+|x|^{-\alpha}) \cdot (|x|^{-\delta}\wedge |x|^{-d-\alpha}) \,, \\
		\label{eq:heatbound4}
		[\psi]_{C^\beta(B_{\ell_x}(x))} & \lesssim (1\wedge|x|)^{-\delta-\beta} \cdot (1\wedge |x|^{-d-\alpha})
		\quad \text{with}\ \ell_x:=1\wedge\tfrac{|x|}{2} \,,\ 0<\beta<\alpha \,.
	\end{align} 
\end{lemma}

The second ingredient prove Theorem~\ref{Thm:1.2density} are bounds for the commutator
\begin{align}\label{eq:commutatorpointwise}
	[|p|^{\alpha},\zeta] v(x) = \mathcal A(d,-\alpha) \int_{\R^d} \frac{\zeta(x)-\zeta(y)}{|x-y|^{d+\alpha}} v(y)\,dy
\end{align}
with appropriate cut-off functions $\zeta\in C^\infty(\R^d)$. When $\alpha<1$, the right side of \eqref{eq:commutatorpointwise} converges pointwise for sufficiently fast decaying $v$; for $\alpha\in[1,2)$, it is understood as principal value integral and can be expressed as in \eqref{eq:commutator1lalpha} below. Concerning the functions $v$ we will assume the size estimate
\begin{equation}
	\label{eq:assv1}
	|v(x)|\leq (1\wedge|x|)^{-\delta} \cdot (1\wedge |x|^{-d-\alpha}) \qquad\text{for all}\ x\in\R^d \, 
\end{equation}
and, for $\beta>\alpha-1$, in addition the regularity estimate
\begin{equation}
	\label{eq:assv2}
	[v]_{C^\beta(B_{\ell_x}(x))} \leq (1\wedge |x|^{-\delta-\beta}) \cdot (1 \wedge |x|^{-d-\alpha})
	\quad \text{for all}\ x\in\R^d\ \text{with}\ \ell_x := 1\wedge \tfrac{|x|}2\,,
\end{equation}
where $\delta\in(-\alpha,\frac{d-\alpha}{2}]$ is now an \emph{arbitrary} parameter.
For two parameters $0<r\ll 1$ and $1 \ll R <\infty$, we choose $\zeta(x)=\chi(x)\cdot\theta(x)$, where
\begin{align}
	\label{eq:asschi1}
	& 0\leq\chi\leq 1 \,,
	\quad
	\chi(x) = 1 \ \text{if}\ |x|\leq R \,,
	\quad
	\chi(x) = 0 \ \text{if}\ |x|\geq 2R \,,
	\quad
	|\nabla\chi|\lesssim R^{-1}\,, \\
	\label{eq:asstheta1}
	& 0\leq\theta\leq 1 \,,
	\quad
	\theta(x) = 0 \ \text{if}\ |x|\leq r \,,
	\quad
	\theta(x) = 1 \ \text{if}\ |x|\geq 2r \,,
	\quad
	|\nabla\theta|\lesssim r^{-1} \,.
\end{align}
If $\alpha\geq1$, we also assume
\begin{align}
	\label{eq:asschi2}
	|D^2\chi| & \lesssim R^{-2}\,, \\
	\label{eq:asstheta2}
	|D^2\theta| & \lesssim r^{-2} \,,
\end{align}
where $D^2\chi$ and $D^2\theta$ denote the Hessians of $\chi$ and $\theta$ respectively.

\begin{lemma}\label{cutoffcomb}
	Let $0<\alpha<2$. Let $0<r\leq 1\leq R<\infty$, assume that $\chi$ and $\theta$ satisfy \eqref{eq:asschi1} and \eqref{eq:asstheta1} and, if $\alpha\geq 1$, also \eqref{eq:asschi2} and \eqref{eq:asstheta2}. Let $-\alpha<\delta\leq\frac{d-\alpha}{2}$, assume that $v$ satisfies \eqref{eq:assv1} and, if $\alpha\geq 1$, also \eqref{eq:assv2} with some $\beta>\alpha-1$. Then
	$$
	\| [(-\Delta)^{\alpha/2},\chi\theta] v \|_{L^2(\R^d)}
	\lesssim r^{-\delta-\alpha+d/2} + R^{-\alpha-d/2} \,.
	$$
\end{lemma}

Once these two lemmas have been proved, we can follow the proof of \cite[Theorem 24]{FrankMerz2023} to deduce that, for any $f\in L^2(\R^d)$, there is a sequence $(\phi_n)\subset C^\infty_c(\R^d\setminus\{0\})$ such that
$$
L_\lambda^{s/2}\phi_n\to f
\ \text{in}\ L^2(\R^d) \,,
$$
and that if, in addition $f\in\dom L_\lambda^{-s/2}$, then the sequence can be chosen such that
$$
\phi_n\to L_\lambda^{-s/2}f
\ \text{in}\ L^2(\R^d) \,.
$$
Once this is shown, we can follow the proof of \cite[Theorem 1]{FrankMerz2023} to show the operator core property claimed in Theorem \ref{Thm:1.2density} and, consequently, to extend the inequalities to the corresponding operator domains. We omit the details of this argument and instead comment on the proofs of the two ingredients mentioned above.

\begin{proof}[Proof of Lemma \ref{pointwise}]
	Inequalities \eqref{eq:heatbound1}, \eqref{eq:heatbound2} and \eqref{eq:heatbound3} follow rather directly from the heat kernel bounds in Theorem~\ref{Thm:2.1}. The bound \eqref{eq:heatbound4} uses Schauder estimates for the fractional Laplacian (see, for instance, \cite[Corollary 2.5]{RosOtonSerra2014}) together with the bounds \eqref{eq:heatbound1} and \eqref{eq:heatbound3}. The arguments are similar (and in fact slightly simpler) to those in \cite[Lemma~26]{FrankMerz2023}.
\end{proof}

\begin{proof}[Proof of Lemma \ref{cutoffcomb}]
	For fixed $v$ as in the lemma, let us set
	\begin{align*}
		I(x) := \int_{\R^d} \frac{\chi(x)-\chi(y)}{|x-y|^{d+\alpha}}v(y)\,dy
		\qquad \text{and} \qquad
		II(x) := \int_{\R^d} \frac{\theta(x)-\theta(y)}{|x-y|^{d+\alpha}}v(y)\,dy.
	\end{align*}
	The assertion of the lemma follows easily from $L^2$-bounds on $I$ and $II$, which we will prove in the following two steps.
	
	\medskip

	\emph{Step 1.} Let $R\geq 1$, assume that $\chi$ satisfies \eqref{eq:asschi1} and, if $\alpha\geq 1$ also \eqref{eq:asschi2}. Let $\delta\leq \tfrac{d-\alpha}2$, assume that $v$ satisfies \eqref{eq:assv1} and, if $\alpha\geq 1$, also \eqref{eq:assv2} with some $\beta>\alpha-1$. Then, we claim
        \begin{align}
          \label{eq:boundI}
          |I(x)| \lesssim \one_{|x|\leq R} R^{-d-2\alpha} + \one_{|x|>R} |x|^{-d-\alpha}
          \qquad\text{for all}\ x\in\R^d \,.
	\end{align}
	In particular,
	$$
	\| I \|_{L^2(\R^d)} \lesssim R^{-\alpha-d/2} \,.
	$$

	The proof of \eqref{eq:boundI} is similar to (and in fact simpler than) that of the analogous bound in \cite[Lem\-ma~20]{FrankMerz2023}. The only difference is that the singularity $x_d^{p-\beta}$ in \cite{FrankMerz2023} is replaced by $|x|^{-\delta-\beta}$ here; but since this singularity is irrelevant for the large length scales involved in the claimed bound, we can argue as in \cite{FrankMerz2023}. The logarithmic terms of \cite{FrankMerz2023} are absent here because the responsible integral \cite[(32)]{FrankMerz2023} is bounded by
	\begin{align*}
		& \one_{\frac R2<|x|\leq4R}\,\frac1R \int_{\frac{R}{4}<|y|\leq8R}\frac{\one_{|x-y|>\ell_x}}{|x-y|^{d+\alpha-1}}|v(y)|\,dy
		\lesssim \one_{\frac R2<|x|\leq4R}\,\frac{1}{R^{d+\alpha}}\int_{\R^d}\frac{dy}{1+|x-y|^{d+\alpha}} \\
		& \quad \lesssim R^{-d-\alpha}\one_{R/2<|x|\leq4R}.
	\end{align*}
	Here we used $|v(y)|\one_{|y|\sim R}\lesssim R^{-d-\alpha}$, $\ell_x\one_{|x|\sim R}=1$, and $|x-y|\one_{|x|\sim|y|\sim R}\lesssim R$.

	\medskip
	
	\emph{Step 2.} Let $r\leq 1$ and assume that $\theta$ satisfies \eqref{eq:asstheta1} and, if $\alpha\geq 1$, also \eqref{eq:asstheta2}. Let $-\alpha<\delta\leq\frac{d-\alpha}{2}$, assume that $v$ satisfies \eqref{eq:assv1} and, if $\alpha\geq 1$, also \eqref{eq:assv2} with some $\beta>\alpha-1$. Then, we claim
	\begin{align}
          \label{eq:boundII}
          |II(x)| \lesssim r^{-\delta-\alpha}\one_{|x|\leq4r} + r^{d-\delta}|x|^{-d-\alpha}\one_{|x|>r/2}
          \qquad\text{for all}\ x\in\R^d \,.
	\end{align}
	In particular,
	$$
	\| II \|_{L^2(\R^d)} \lesssim r^{-\delta-\alpha+d/2} \,.
	$$

	The proof of \eqref{eq:boundII} is similar to that of \eqref{eq:boundI} in Step 1, but requires a slightly more careful consideration of the singularity at the origin. Therefore we include some details.

	\emph{Case $\alpha<1$}. It is easy to see that
  \begin{align}\label{eq:radial1}
    |II(x)| & \lesssim \one_{|x|\leq 4r} \int_{|y|>r} \frac{1}{|y|^{d+\alpha}} |v(y)| \,dy + \one_{|x|>\tfrac r2} \frac{1}{|x|^{d+\alpha}} \int_{|y|\leq 2r} |v(y)| \,dy \notag \\
           & \quad + \one_{\tfrac r2<|x|\leq 4r} \frac 1r \int_{\tfrac r4<|y|\leq 8r} \frac{1}{|x-y|^{d+\alpha-1}} |v(y)| \,dy \,.
  \end{align}
  Inserting the bounds on $v$ into the right side of \eqref{eq:radial1} yields
  $$
  \int_{|y|\leq 2r} (|y|^{-\delta}\wedge |y|^{-d-\alpha}) \,dy \lesssim r^{d-\delta}
  $$
  and, since $\delta>-\alpha$,
  $$
  \int_{|y|>r} \frac1{|y|^{d+\alpha}} (|y|^{-\delta}\wedge |y|^{-d-\alpha}) \,dy \lesssim r^{-\delta-\alpha} \,.
  $$
  Finally, if $\tfrac r2<|x|\leq 4r$, then, since $\alpha\in(0,1)$,
  \begin{align*}
    \int_{\tfrac r4<|y|\leq 8r} \frac{1}{|x-y|^{d+\alpha-1}} (|y|^{-\delta} \wedge |y|^{-d-\alpha}) \,dy
    \lesssim r^{-\delta} \int_{\tfrac r4<|y|\leq 8r} \frac{dy}{|x-y|^{d+\alpha-1}}
    \lesssim r^{-\delta-\alpha+1} \,.
  \end{align*}
  Here we replaced the integral over $\{\tfrac r4<|y|\leq 8r\}$ by the integral over $|x-y|\leq 12r$. This proves the claimed pointwise bound. The $L^2$-bound follows by an integration.

	\medskip
	
	\emph{Case $\alpha\geq 1$}. We fix a local length scale $\ell_x$, depending on $x\in\R^d$, and decompose
\begin{align}
  \label{eq:commutator1lalpha}
  \begin{split}
    II(x)
    & = \int_{|y-x|\leq\ell_x} \frac{\theta(x)-\theta(y)}{|x-y|^{d+\alpha}} (v(y)-v(x))\,dy \\
    & \quad + v(x) \int_{|y-x|\leq\ell_x} \frac{\theta(x)-\theta(y)+\nabla\theta(x)\cdot(y-x)}{|x-y|^{d+\alpha}} \,dy \\
    & \quad + \int_{|y-x|>\ell_x} \frac{\theta(x)-\theta(y)}{|x-y|^{d+\alpha}} v(y)\,dy \,.
  \end{split}
\end{align}
Note that because of the principal value we were free to introduce the term $\nabla\theta(x)\cdot(y-x)$, which contributes zero to the integral (because of oddness), but makes it converge absolutely. We bound the first term by
\begin{align}\label{eq:commutator1lalpha1}
  \begin{split}
    & \left| \int_{|y-x|\leq\ell_x} \frac{\theta(x)-\theta(y)}{|x-y|^{d+\alpha}} (v(y)-v(x))\,dy \right| \\
    & \quad \leq [v]_{C^\beta(B_{\ell_x}(x))} [\theta]_{C^1(B_{\ell_x}(x))} \int_{|y-x|\leq\ell_x} \frac{dy}{|x-y|^{d+\alpha-1-\beta}} \,dy \\
    & \quad \lesssim [v]_{C^\beta(B_{\ell_x}(x))} [\theta]_{C^1(B_{\ell_x}(x))} \ell_x^{-\alpha+1+\beta}
  \end{split}
\end{align}
for some $\beta>\alpha-1$. Similarly, we bound the second term by
\begin{align}\label{eq:commutator1lalpha2}
  \begin{split}
    & \left| v(x) \int_{|y-x|\leq\ell_x} \frac{\theta(x)-\theta(y)+\nabla\theta(x)\cdot(y-x)}{|x-y|^{d+\alpha}} \,dy \right| \\
    & \quad \leq |v(x)| [\theta]_{C^2(B_{\ell_x}(x))} \int_{|y-x|\leq\ell_x} \frac{dy}{|x-y|^{d+\alpha-2}} \,dy
      \lesssim |v(x)| [\theta]_{C^2(B_{\ell_x}(x))} \ell_x^{2-\alpha} \,.
  \end{split}
\end{align}

  For the first term in \eqref{eq:commutator1lalpha} we use the bound \eqref{eq:commutator1lalpha1} and note that $[\theta]_{C^1(B_{\ell_x}(x))}$ vanishes unless $|x|\sim r$, in which case it is $\lesssim r^{-1}$. This leads to a bound
  $$
  \one_{|x|\sim r} (1\wedge |x|^{-d-\alpha}) (1\wedge |x|)^{-\delta-\alpha+1} r^{-1} \sim \one_{|x|\sim r}\,r^{-\delta-\alpha} \,. 
  $$
  Similarly, for the second term in \eqref{eq:commutator1lalpha} using the bound \eqref{eq:commutator1lalpha2} we obtain
  $$
  \one_{|x|\sim r} (1\wedge |x|^{-d-\alpha}) (1\wedge |x|)^{-\delta-\alpha+2} r^{-2} \sim \one_{|x|\sim r}\,r^{-\delta-\alpha} \,.
  $$
  
  We now turn to the third term in \eqref{eq:commutator1lalpha},
  $$
  \widetilde{II}(x) := \int_{|y-x|>\ell_x} \frac{\theta(x)-\theta(y)}{|x-y|^{d+\alpha}} v(y)\,dy \,,
  $$
  which is bounded by
  \begin{align}\label{eq:radial2}
    \left| \widetilde{II}(x) \right| & \lesssim \one_{|x|\leq 4r} \int_{|y|>r} \frac{1}{|y|^{d+\alpha}} |v(y)| \,dy + \one_{|x|>\tfrac r2} \frac{1}{|x|^{d+\alpha}} \int_{|y|\leq 2r} |v(y)|\,dy \notag \\
                                    & \quad + \one_{\tfrac r2<|x|\leq 4r} \frac1r \int_{\tfrac r4<|y|\leq 8r} \frac{\one_{|x-y|>\ell_x}}{|x-y|^{d+\alpha-1}}|v(y)|\,dy \,.
  \end{align}

  We now insert the bounds on $v$ into the right side of \eqref{eq:radial2}. The first two terms are bounded as in the case $\alpha<1$. To bound the third term in \eqref{eq:radial2}, we use $|v(y)|\one_{|y|\sim r}\lesssim r^{-\delta}$, $\ell_x=|x|/2\geq r/4$ for $|x|\geq r/2$, and $r/4<|x-y|\in[r/4,12r]$, and obtain
  \begin{align*}
    & \int_{\tfrac r4<|y|\leq 8r} \frac{\one_{|x-y|>\ell_x}}{|x-y|^{d+\alpha-1}}|v(y)|\,dy
      \lesssim r^{-\delta+1-\alpha}.
  \end{align*}
  
  Combining all these bounds, we obtain the claimed pointwise bound. As before, the $L^2$-bound follows by an integration.
\end{proof}



\newcommand{\etalchar}[1]{$^{#1}$}
\def\cprime{$'$}
\providecommand{\bysame}{\leavevmode\hbox to3em{\hrulefill}\thinspace}
\providecommand{\MR}{\relax\ifhmode\unskip\space\fi MR }
\providecommand{\MRhref}[2]{%
  \href{http://www.ams.org/mathscinet-getitem?mr=#1}{#2}
}
\providecommand{\href}[2]{#2}

\end{document}